\documentclass[letterpaper]{amsart}
\usepackage{amsmath}
\usepackage{amssymb}
\usepackage{epsfig}
\usepackage[all, knot]{xy}
\xyoption{arc}

\newtheorem{thm}{Theorem}[section]
\newtheorem*{thmA}{Theorem~A}
\newtheorem*{thmM}{Main Theorem}

\newtheorem{df}[thm]{Definition}
\newtheorem{prop}[thm]{Proposition}
\newtheorem{cor}[thm]{Corollary}
\newtheorem{lem}[thm]{Lemma}
\newtheorem{ex}[thm]{Example}
\newtheorem{rem}[thm]{Remark}

\newtheorem{claim}{Claim}

\newcommand{\Pre}{\operatorname{Preper}}

\newcommand{\Per}{\operatorname{Per}}
\newcommand{\pp}{\mathbb{P}}
\newcommand{\cc}{\mathbb{C}}
\newcommand{\qq}{\mathbb{Q}}
\newcommand{\af}{\mathbb{A}}
\newcommand{\ox}{\mathcal{O}}
\newcommand{\Rat}{\operatorname{Rat}}
\newcommand{\Mor}{\operatorname{Mor}}
\newcommand{\PGL}{\operatorname{PGL}}

\newcommand{\Com}{\operatorname{Com}}

\begin{document}

\title
{Finiteness of commutable maps of bounded degree}
\author{Chong Gyu Lee}

\author{Hexi Ye}



\date{\today}

\address{Department of Mathematics, Statistics and Computer Science, University of Illinois at Chicago, Chicago, IL, 60607, US}

\email{phiel@math.uic.edu}

\address{Department of Mathematics, Statistics and Computer Science, University of Illinois at Chicago, Chicago, IL, 60607, US}

\email{hye4@uic.edu}

\maketitle

\begin{abstract}
     In this paper, we study the relation between two dynamical systems $(V,f)$ and $(V,g)$ with $f\circ g = g \circ f$. As an application, we show that an endomorphism (respectively a polynomial map with Zariski dense, of bounded $\Pre(f)$) has only finitely many endomorphisms (respectively polynomial maps) of bounded degree which are commutable with $f$.
\end{abstract}

\section{Introduction}

     A dynamical systems $(V,f)$ consists of a set $V$ and a self map $f:V \rightarrow V$. If $V$ is a subset of a projective space $\pp^n$ defined over a finitely generated field $K$ over $\mathbb{Q}$, then we have arithmetic height functions on $V$, which make a huge contribution to the study of $(V,f)$. In particular, if $f$ is an endomorphism on $\pp^n$ or regular polynomial automorphism, then we have lots of results, basic height equality \cite{N, Le, K2}, the canonical height function \cite{CS, K2} and the equidistribution of periodic points \cite{Y}.

     In this paper, we show the following interesting results by studying arithmetic relations between two dynamical systems $(V,f)$ and $(V,g)$ with $f \circ g = g \circ f$.
    \begin{thmM}[Theorem~\ref{thmA}, \ref{thmB}]
    Let $M$ be the set of endomorphisms on $\pp^n$ (respectively polynomial maps on $\af^n$), defined over $\cc$. Suppose that we pick $f \in M$ such that $\Pre(f)$ is of bounded height and Zariski dense. Then,
    \[
    \Com(f,d): = \{ g \in M ~|~ f\circ g =  g \circ f, ~\deg g = d\}
    \]
    is a finite set.
    \end{thmM}

    Theorem~M for the morphism case is a generalization of some known result. If $d=1$, $\Com(f,1)$ is called the group of automorphisms of $f$ (See Corollary~\ref{autoend}). When $d>1$, Theorem M is a generalization of the result of Kawaguchi and Silverman \cite{KS}: they proved that a sequence of endomorphisms $\langle \phi_m \rangle$, defined over $\overline{\qq}$, sharing the same canonical height must satisfy either $\deg \phi_m$ or $[\qq(\phi_m):\qq]$ goes to infinity. ($\mathbb{Q}(\phi_m)$ is a field generated by all coefficient of $\phi_m$ over $\qq$.) In particular, endomorphisms commuting each other have the same canonical height function and hence we can see that a sequence of endomorphisms which commute with a given endomorphism must satisfy $\lim \deg \phi_m = \infty$. Also, Theorem~M is related to known results of Ritt \cite{R} and Dinh-Sibony \cite{Di, DS}, the classification of commutable polynomial maps. The main theorem supports Dinh-Sibony's classification: each class only contain finitely many element of bounded degree.

    This paper has two main ingredients: the first one is the relation between the sets of preperiodic points of commutable maps $f$ and $g$. In Section~2, we study the relation between $(V,f)$ and $(V,g)$ when $f,g$ are commutable. In particular, if $f,g$ are commutable, then $g$ becomes a self map on $\Pre(f)$ preserving the exact period. (See Proposition~\ref{cm}.)

    The second ingredient is a finite set of points in general position of degree $d$: it is a generalization of points in general position for general linear maps. We can find a finite set of points which exactly determine a rational map on $\pp^n$ of fixed degree $d$:

    \begin{thmA}
        Let $d$ be a positive integer. Then, there is a set $S_d$ of finite points such that for any endomorphism on $\pp^n$ (respectively polynomial map on $\af^n$) of degree $d$ is exactly determined by image of $S_d$.
    \end{thmA}

    Theorem~\ref{det2} claims that we can find points in general position of any degree if $\Pre(f)$ is Zariski dense. Moreover, if $\Pre(f)$ is of bounded height, then we only have finitely many preperiodic points of given period. Therefore, we only have a finite set of points which contains image of $S_d$ and hence we can complete the proof of the Main Theorem.

    In case of endomorphism, we know $\Pre(f)$ is of bounded height and Zariski dense, thanks to Call-Silverman \cite{CS} and Fakhruddin \cite{Fa}. For polynomial maps, some special examples are known - regular polynomial automorphism \cite{De, M, Le8} and Cartesian product of polynomials of degree at least 2 \cite{MS}. We introduce new example, polynomial maps of small $D$-ratio. (See Theorem~\ref{Zdense}.)

    Thanks to Thomas Scanlon, we have another proof for the endomorphism case using the space of endomorphisms on $\pp^n$. We use the original proof for Theorem~\ref{thmB} and introduce Scanlon's proof for Theorem~\ref{thmA}.

\par\noindent\emph{Acknowledgements}.\enspace
We would like to thank Thomas Scanlon, especially for his help in Section~3. We also thank Laura DeMarco and Shu Kawaguchi for her helpful discussions and comments.

\section{Commuting maps and their dynamical systems}

    In this section, we study dynamical systems of commutable maps, focused on the relation between the sets of preperiodic points of commutable maps. We start with the definition of dynamical sets.
    \begin{df}
        Let $f : V \rightarrow V$. We define the following dynamical sets.
        \[
        \begin{array}{l}
        \ox_f(P): = \{ f^m(P) ~|~ m=0, 1,2, \cdots \} \\
        \Per_m(f): = \{ P \in V ~|~ f^m(P) = P\} \\
        \Per(f) :=  \displaystyle \bigcup_{m=1}^\infty \Per_m(f) \\
        \Pre_{m,l}(f) := \{ P\in V ~|~ f^{m}(P) = f^{l}(P) ~\text{for some}~m>l\geq 0\} \\
        \Pre(f) : = \displaystyle \bigcup_{m=1}^\infty \bigcup_{l=0}^{m-1} \Pre_{m,l} (f)
        \end{array}
        \]
    \end{df}

    The commutativity is a very strong condition. Without height functions on $V$, we can find information of dynamical systems of commutable maps.
    \begin{prop}\label{cm}
        Let $f,g:V \rightarrow V$ be maps on a set $V$ such that $ f \circ g = g \circ f$. Then the followings are true.
        \begin{enumerate}
            \item $g\bigl ( \ox_f(P) \bigr) = \ox_f \bigl( g(P) \bigr)$
            \item $g\bigl( \Pre_{m,l}(f) \bigr) \subset \Pre_{m,l}(f)$, $g\bigl( \Pre(f) \bigr) \subset \Pre(f).$
            \item If $g$ is a finite-to-one map on an infinite set $V$ or $g$ is surjective map on a finite set $V$, then $g\bigl( \Pre(f) \bigr) = \Pre(f).$
        \end{enumerate}
    \end{prop}
    \begin{proof}
        \begin{enumerate}
            \item $g\bigl ( \ox_f(P) \bigr) = \{ g\bigl( f^m(P) \bigr) \} =  \{ f^m \bigl( g (P) \bigr) \}= \ox_f \bigl( g(P) \bigr)$
            \item If $f\in \Pre_{m,l}(P)$, then $f^l(P) = f^{m}(P)$. So, we get
            \[
             f^l\bigl( g(P) \bigr) =  g\bigl(f^l(P)\bigr)= g\bigl( f^{m}(P) \bigr) = f^{m}\bigl(g(P)\bigr)
            \]
            and hence $g(P) \in \Pre_{m,l}(P)$. Moreover,
            \[
            g\bigl( \Pre(f) \bigr) \subset \bigcup_{m=1}^\infty \bigcup_{l=0}^{m-1} g\bigl( \Pre_{m,l}(f) \bigr)
             \subset \bigcup_{m=1}^\infty \bigcup_{l=0}^{m-1} \Pre_{m,l}(f) = \Pre(f)
            \]
            \item If $V$ is a finite set, then $\Pre(f) = V$ for any map $f$. Therefore, $g\bigl( \Pre(f) \bigr) = \Pre(f)$ means
            $g(V) = V$.

            Let $V$ be an infinite set and let $g$ be a finite-to-one self map on $V$. Suppose that $g(P) \in \Pre(f)$ and $P \not\in \Pre(f)$. Then, $g$ is a map from infinite set to finite set:
            \[
            g : O_f(P) \rightarrow O_f\bigl ( g(P) \bigr),
            \]
            which is a contradiction.
        \end{enumerate}
    \end{proof}

    It is clear that more information of $V$ we have, more information of $(V,f)$ we can get. If we have arithmetic information like the finiteness of $K$-rational preperiodic points, we have the following result.
    \begin{prop}\label{cm2}
        Let $K$ be a finitely generated field over $\mathbb{Q}$, let $V$ be a subset of a projective variety defined over $K$ and let $f,g:V \rightarrow V$ be maps on a set $V$ such that $f \circ g = g \circ f$. Suppose that $\Pre(f)\cap V(K')$ is finite for all finite extension $K'$ of $K$. Then, $\Pre(f) \subset \Pre(g).$
    \end{prop}
    \begin{proof}
            Pick a point $P \in \Pre(f)$. Then, $g^k(P) \in \Pre(f)$ because of Proposition~\ref{cm},$(2)$.
             Suppose that $K'$ is a finite extension of $K$, containing all coordinate of $P$ and all coefficient of $f$.
             Then, $g^k(P) \in \Pre(f) \cap V(K')$. Therefore, $\ox_g(P) \subset \Pre(f) \cap V(K')$ and hence $\ox_g(P)$ is a finite set.
    \end{proof}

    The condition of Proposition~\ref{cm2} is a common one in arithmetic dynamics. Let $h=h^{\overline{B}}_L$ be an arithmetic height function on $\overline{K}$, associated with a nef and big polarization of $B$ and an ample line bundle $L$ on $\pp^n$. (See \cite{MO} for details.) If $\phi$ is an endomorphism or regular polynomial automorphism of degree $d>1$ then $\Pre(\phi)$ is of bounded height and hence $\Pre(\phi)$ satisfies the condition of Proposition~\ref{cm2}.

    \begin{ex}
        Let $f(x,y) = (x^2,y^2)$ and $g(x,y) = (x,xy)$. Then, they are commutable with each other.
        Moreover, $f$ is finite-to-one map and $\Pre(f)$ is a set of bounded height. Thus, by Proposition~\ref{cm} and \ref{cm2}, we know that
        \[
        g \Pre(f) \subset \Pre(f), \quad \Pre(f) \subset \Pre(g).
        \]
        On the other hand, we know that
        \[
        g(\af^n) \cap \left( \{0\} \times \af^1 \right) = \{(0,0)\}
        \]
        and hence $(0, 1) \in \Pre(f) \setminus g\bigl( \Pre(f) \bigr)$. Also, $(0,2) \in \Pre(g) \setminus \Pre(f)$.  Therefore,
        \[
        g \Pre(f) \varsubsetneq \Pre(f), \quad \Pre(f) \varsubsetneq \Pre(g).
        \]
    \end{ex}
    \begin{ex}
        Let $V$ be a projective variety over a number field and let $f,g:  V \rightarrow V$ be endomorphisms commutable with each other.
        Assume that $f,g$ are polarized by the same ample divisor $D$: $f^*D \sim q_fL, g^*D \sim q_gL$ for some $q_f, q_g >1$.
        Then, they satisfy all condition of Proposition~\ref{cm2}. So, $\Pre(f) = \Pre(g)$ and hence they have the same canonical height:
        $\widehat{h}_f = \widehat{h}_g$ by \cite[Theorem 3.1.2]{YZ}. With the same reason, $\Pre(f^m\circ g^l) = \Pre(f)$ and $\widehat{h}_{f^m\circ g^l} = \widehat{h}_{f}$ for all $(m,l) \neq (0,0)$.
    \end{ex}

    \begin{thm}
        Let $K$ be a finitely generated field over $\mathbb{Q}$, let $f,g: \af^n \rightarrow \af^n$ be polynomial maps commutable with each other and let $r(f), r(g)$ be the $D$-ratio of $f$ and $g$ defined on \cite{Le2}. Suppose that $r(f) < \deg f $ and $r(g) < \deg (g)$. Then,
    \[
    \Pre(f)= \Pre(g) = \Pre(f^m \circ g^l)
    \]
    for all $(m,l) \neq (0,0)$.
    \end{thm}
    \begin{proof}
        \cite[Theorem 5.3]{Le2} says that $\Pre (f)$ is a set of bounded height if $r(f) < \deg(f)$. Thus,
    $f,g$ satisfy the condition of Proposition~\ref{cm}~(3). Thus, $\Pre(f)=\Pre(g)$.

    For the second equality,  we know that $f^m \circ g^l$ also commute with $f$ and $g$ and hence $\Pre(f) \subset \Pre(f^m\circ g^l)$. Furthermore,  \cite[Proposition 4.5]{Le2} says that
    \[
    \dfrac{r(f\circ g)}{\deg f \circ g} \leq \dfrac{r(f)}{\deg f} \cdot \dfrac{r(g) }{\deg g}.
    \]
    By induction, we have
    \[
    \dfrac{r(f^m \circ g^l) }{\deg f^m \circ g^l} \leq \left(\dfrac{r(f)}{\deg f}\right)^m \cdot \left( \dfrac{r(g) }{\deg g} \right)^l.
    \]
    Thus, the assumption $\dfrac{r(f)}{\deg f}, \dfrac{r(g) }{\deg g}<1$ derives $r(f^m \circ g^l) < \deg (f^k \circ g^l) $. Therefore, $\Pre(f^m \circ g^l) $ is of bounded height and hence $\Pre(f^k\circ g^l)\subset  \Pre(f)$.
    \end{proof}

    In Proposition~\ref{cm}, we showed that $g$ becomes a self map on $\Per_{m,l}(f)$. However, $g$ does not preserve the period. Here's a clue to check the change of the period by observing the multiplier of a periodic point.

    \begin{df}
        Let $f$ be an endomorphism on $\pp^1(\cc)$ and $P \in \Per(f)$ be a point with period $l$. We define \emph{the multiplier of $P$ by $f$} to be
        \[
        \lambda_f(P) := (f^l(P))' = \prod_{i=1}^l f'\bigl( f^i(P) \bigr)
        \]
        where $l$ is the period of $P$.
    \end{df}

    we generalize the multiplier in higher dimension. Though it cannot carry much information as in dimension1, it is enough for my purpose.

    \begin{df}
        Let $f$ be an endomorphism on $\pp^n(\cc)$ and let $P \in \Per(f)$ be a point with period $l$. We define \emph{the multiplier of $P$ by $f$} to be
        \[
        \lambda_f(P) :=
        \left| \dfrac{df^l}{dX}(P) \right|= \prod_{i=1}^l \left| \dfrac{df}{dX}\bigl( f^i(P) \bigr) \right|
        \]
        where $l$ is the period of $P$ and $\dfrac{df}{dX} = \left( \dfrac{\partial f_i}{\partial x_j}\right)$ is the Jacobian matrix of $f$.
    \end{df}

    \begin{rem}
        The multiplier is $\PGL_{n+1}$-invariant; suppose that $P\in \Per(f)$ of period $l$ and $\sigma \in \PGL_{n+1}$. Then $Q=\sigma(P)$ is a periodic point of $f^{\sigma} = \sigma \circ f \circ \sigma^{-1}$ of period $l$. Then,
        \begin{eqnarray*}
        \lambda_{f^{\sigma}} (Q) &= &\left| \dfrac{d(\sigma \circ f \circ \sigma^{-1})^l}{dX}(Q) \right| \\
        &= &\left| \dfrac{d(\sigma \circ f^l \circ \sigma^{-1})}{dX}(Q) \right| \\
         &= &\left| \dfrac{d\sigma}{dX} \bigl( f^l \circ \sigma^{-1}(Q) \bigr)
                    \cdot \dfrac{d f^l }{dX}\sigma^{-1}(Q)
                    \cdot \dfrac{d \sigma^{-1}}{dX}(Q)\right| \\
         &= &\left| \dfrac{d\sigma}{dX}(P)
                    \cdot \dfrac{d \sigma^{-1}}{dX}(Q)
                    \cdot \dfrac{d f^l }{dX}(P)\right| \\
         &= &\left| \dfrac{d f^l }{dX}(P)\right| \\
          &= & \lambda_f (P).
        \end{eqnarray*}
    \end{rem}

    \begin{df}
        Let $f :\cc^n \rightarrow \cc^n$ be a polynomial map and $P \in \Per(f)$. We say that \emph{$P$ is a critical point of $f$} if $\left| \dfrac{df}{dX}(P)\right|=0 $.
    \end{df}

    \begin{prop}\label{multi}
        Let $f,g :\cc^n \rightarrow \cc^n$ be polynomial maps commutable with each other and let $P \in \Per(f)$. Suppose that $P$ is not critical point of $g$. Then
        \[
        \lambda_f(P) = (\lambda_f\bigl(g(P)\bigr)^{l_0}
        \]
        where $l_0 = \dfrac{f\text{-period of }P}{f\text{-period of }g(P)}$.
    \end{prop}
    \begin{proof}
    From the assumption, we have $g \circ f^l = f^l \circ g$ for all $l\geq 0$. By the Chain rule, we get
    \[
    {\dfrac{dg}{dX}} \bigl( f^l(X) \bigr) \cdot {\dfrac{d(f^l)}{dX}}(X) = {\dfrac{d(f^l)}{dX}}\bigl( g(X) \bigr) {\dfrac{dg}{dX}}(X).
    \]
    Suppose that $P$ is a periodic points of $f$ with period $l$. Then,
    \[
    {\dfrac{dg}{dX}}(P) {\dfrac{d(f^l)}{dX}}(P)
    = {\dfrac{dg}{dX}}(f^l(P)) {\dfrac{d(f^l)}{dX}}(P)
    = {\dfrac{d(f^l)}{dX}}(g(P)) {\dfrac{dg}{dX}}(P).
    \]
    Thus, if $P$ is not critical point of $g$, we have
    \[
    \lambda_f\bigl( g(P) \bigr) = (\lambda_f(P))^{1/l'}.
    \]
    \end{proof}

    \begin{thm}
    Let $f,g : \cc^n \rightarrow \cc^n$ be polynomial maps such that $f\circ g = g \circ f$, let $P$ be a $f$-periodic point of period $l$ and let $K = \qq(f,g,P)$.
    If $\lambda_f(P)$ does not have $l_0$-th root in $K$ for all divisor $l_0$ of $l$, then $g$ preserves the $f$-period of $P$.
    \end{thm}
    \begin{proof}
    By definition of multiplier, $L$ should contain both $\lambda_f(P), \lambda_f\bigl(g(P)\bigr)$.
    If $g$ does not preserve the $f$-period of $P$, then
    \[
    \lambda_f(P)= \lambda_f(g(P))^{l_0}
    \]
    by Proposition~\ref{multi}.
    \end{proof}

    \begin{ex}[Finding rational periodic points of the exact period]
    Let $f(x,y) = (y^2, x^2)$, $g(x,y) = (x^2y, xy^2)$. Then $f,g$ are commutable with each other. Let $\zeta = \zeta_7$. Then $P = (\zeta, \zeta^2)$ is a periodic point of $f$ of period $3$, of multiplier $\lambda_f(P) = -4^3 \zeta^6$. Since $-1$ is not cube in $\mathbb{Q}(\zeta)$ and $P$ is not critical points of $g$, we can easily claim that all points in
    \[
    \ox_g(P) =\{(\zeta^4, \zeta^5), (\zeta^{-1}, 1), (\zeta^{-2}, \zeta^{-1})\}
    \]
     consists of periodic points of period 6. Similarly, use $g_2(x,y) = (xy^2, x^2y)$ to get other points of period $6$:
    \[
    (\zeta^5, \zeta^4), (\zeta^{-1}, 1), (\zeta^{-1}, \zeta^{-2}) \in \ox_{g_2}(P).
    \]
    \end{ex}

\section{The set of endomorphisms commutable with an endomorphism $\phi$}

    In this section, we prove Theorem~M for endomorphism case using the space of rational maps and the space of endomorphisms on $\pp^n$: we consider $\Rat_d^n$, the set of rational maps of degree $d$ on $\pp^n$. Then, it is a subset of $\pp^N$ for some $N$.
    \[
    \iota : \Rat_d^n \to \pp^N  \quad \phi:=[\phi_0, \cdots, \phi_n] \mapsto [a_{iJ}]_{i=0, \cdots, n, J\in \mathcal{J}}
    \]
    where $\mathcal{J} = \{J = \langle j_0, \cdots, j_n  \rangle ~|~ \sum j_i = d \}$ and $f_i = \sum_{J\in \mathcal{J}} a_{iJ} X^J$ and $X^J = X_0^{j_0} \cdots X_n^{j_n}$. We also consider $\Mor_d^n$, the set of endomorphisms of degree $d$ on $\pp^n$, in $\pp^N$. Then, $\Mor_d^n$ is a Zariski open subset of $\pp^N$ because it is the inverse image of $\cc^*$ by the Resultant map. (See \cite{LV} or \cite[\S4.2]{S2}.)

    Now, choose an endomorphism $\phi$ and consider $\Com(\phi,d)$, a subset of $\Mor_d^n$:
    \[
    \Com(\phi,d) \subset \Mor_d^n \subset \pp^{N}, \quad [a_{0I} X^I, \cdots, a_{nI} X^I] \mapsto [a_{0I}, \cdots, a_{nI}]_{I\in \mathcal{I}}
    \]
    $\mathcal{I} = \{(I_1, \cdots, I_n) ~|~ \sum_{t=1}^n I_t = d\}.$ Next lemma says that $\Com(\phi,d)$ is a Zariski closed subset of $\Rat_d^n$ so that $\Com(\phi,d)$ has finitely many irreducible component.

    \begin{lem}\label{closed}
    Let $\phi: \pp^n \rightarrow \pp^n$ be an endomorphism. Then $\Com(\phi,d)$ is a Zariski closed subset of $\Mor_d^n$.
    \end{lem}
    \begin{proof}
        Let $\phi = [\phi_0 , \cdots , \phi_n], \psi= [\psi_0 , \cdots , \psi_n]$ be endomorphisms on $\pp^n$ with $\psi \in \Com(\phi,d)$. Suppose
        \[
        \phi_i = \sum_\mathcal{J'} a_{iJ'} X^{J'}, \quad \psi_i = \sum_\mathcal{J} b_{iJ} X^J
        \]
        where
        $\mathcal{J'} = \{(j_0, \cdots, j_n) ~|~ \sum_{i=0}^n j_i = \deg \phi\},
        \mathcal{J} = \{(j_0, \cdots, j_n) ~|~ \sum_{i=0}^n j_i = d \}$ and $X^J = x_1^{j_0} \cdots x_n^{j_n}$.

        Let
        \[
        \phi \circ \psi = [(\phi \circ \psi)_0, \cdots, (\phi \circ \psi)_0].
        \]
        $\psi\in \Com(\phi,d)$ if and only if there is a rational function $f(X)$ such that $(\psi \circ \phi)_i = f(X) (\phi \circ \psi)_i$ for all $i=0, \cdots, n$. So, $(\psi \circ \phi)_0 \cdot (\phi \circ \psi)_i=  (\psi \circ \phi)_i \cdot (\phi \circ \psi)_0$. By comparing coefficients of each monomial of degree $d\cdot \deg \phi$, we have a system of homogeneous equations which determines a closed subset of $\Mor_d^n$. Thus, $\Com(\phi,d)$ is a finite intersection of Zariski closed subset of $\Mor_d^n$.
    \end{proof}

    \begin{thm}\label{irre}
    Let $\phi:\pp^n \rightarrow \pp^n$ be an endomorphism of degree $d>1$ and let $\{[\psi]_t~|~t \in T\}$ be an irreducible family in $\Mor_d^n$ such that $\phi \circ [\psi]_t = [\psi]_t \circ \phi$. Then, $[\psi]_s = [\psi]_t$ for all $s,t \in T$.
    \end{thm}
    \begin{proof}
    Since $\phi$ is an endomorphism of degree $d>1$, $\Per_m(\phi) = \Pre_{m,0}(\phi)$ is a finite set \cite[Theorem~3.1]{FS} and hence discrete. By Proposition~\ref{cm}, $(\psi)_t$ maps $\Per_m(\phi)$ to itself. Thus, we have a map
    \[
    \psi : \Per_{m}(\phi) \times T \rightarrow \Per_{m}(\phi).
    \]

    Choose a point $P\in \Per_{m}(\phi)$. Then we have continuous map from a connected set to totally discrete set:
    \[
    \psi(\cdot, P) :  T \rightarrow \Per_{m}(\phi).
    \]
    Therefore, $\psi(\cdot, P)$ is a constant map on every $\Per_m(\phi)$ and hence on $\Per(\phi)$. Becasue $\Per(\phi)$ is Zariski dense \cite[Theorem~5.1]{Fa}, irreducible family $\{[\psi]_t ~|~ t \in T\}$ has the same image on Zariski dense set $\Pre(\phi)$ and hence every member is the same map.
    \end{proof}

    Theorem~\ref{irre} is enough to prove Theorem~\ref{thmA}:

    \begin{thm}\label{thmA}
    Let $\phi:\pp^n \rightarrow \pp^n$ be an endomorphism of degree at least $2$, defined over $\cc$. Then,
    \[
    \Com(\phi,d): = \{ \psi :\pp^n \rightarrow \pp^n ~|~ \psi \circ \phi =  \phi \circ \psi, ~\deg g = d\}
    \]
    is a finite set.
    \end{thm}
    \begin{proof}
    By Lemma~\ref{closed}, $\Com(\phi,d)$ is a subvariety of $\Mor^n_d$ so that it consists of finitely many irreducible component. Theorem guarantees that we only have a unique map from each irreducible component so that $\Com(\phi,d)$ is a finite set.
    \end{proof}
    \begin{cor}
    Let $\phi:\pp^n \rightarrow \pp^n$ be an endomorphism of degree at least $2$ and $\{(\psi)_t:\pp^n \rightarrow \pp^n\}$ be a sequence of distinct endomorphisms such that
    $(\psi)_t \circ \phi = \phi  \circ (\psi)_t$ for all $t$. Then, $\deg (\psi)_t \to \infty$.
    \end{cor}

    \begin{cor}\label{autoend}
        Let $\phi:\pp^n \rightarrow \pp^n$ be an endomorphism of degree at least $2$. Then,
        \[
        \operatorname{Aut}(\phi) = \{ \eta \in \operatorname{PGL}_{n+1}(\cc) ~|~ \eta \circ \phi \circ \eta^{-1} = \phi \}
        \]
        is a finite set.
    \end{cor}
    \begin{proof}
    It is true if $\eta$ is invertible:
    \[
    \eta \circ \phi \circ  \eta^{-1}  \quad \Leftrightarrow \quad  \eta \circ \phi = \phi \circ \eta.
    \]
    Thus, $\operatorname{Aut}(\phi) = \Com(\phi,1)$.
    \end{proof}

\section{Points in general position of degree $d$}

    In this section, we introduce another way of proving finiteness of commutable maps: we say that finitely many points $P_1, \cdots, P_N$ in $\pp^n$ are in general position if none of hyperplane of $\pp^n$ contains all $P_1, \cdots P_N$. It is well known fact that a projective linear map on $\pp^n$ is uniquely determined by images of points in general position. We generalize it for polynomial maps of higher degree to prove the polynomial map case for Theorem~A.

    \begin{thmA}[Points in general position of degree $d$ for polynomial maps]\label{det2}
        Let $d$ be a positive integer. Then, there is a set $S_d$ of finite points such that for any polynomial map $g: \cc^n \rightarrow \cc^n$ of degree $d$ is exactly determined by image of $S_d$.
    \end{thmA}
    \begin{proof}
        Let $N$ be the number of monomials of degree of $n$ variables, whose degree is at most $d$. Consider the Veronese map
        \[
               \tau: \cc^n \rightarrow \pp^{N-1}(\cc),
               \quad (X_1, \cdots, X_n) \mapsto [1, X_1, \cdots, X_1^{d_{0I}}\cdots X_n^{d_{nI}}, \cdots, X_n^d].
        \]
        Remark that $\tau$ takes all possible monomials, of $n$ variables, of degree at most $d$. $\tau$ is a map between $\cc^n$ and $\cc^{N}$ so that the image of $\tau$ is a closed set.
        \begin{claim}\label{c1}
        $\tau(\cc^n)$ is not contained in a finite union of hyperplanes.
        \end{claim}
        \begin{proof}[Proof of Claim~\ref{c1}]
        Suppose that $\tau(\cc^n)$ is contained in a finite union of hyperplane. Then, the preimage of a hyperplane by $\tau$ is the zero set of a linear combination of monomials of degree at most $d$ so that the preimage of a hyperplane is a union of algebraic sets of degree at most $d$. However, $\cc^n$ is not covered by finitely many subvarieties of degree at most $d$.
        \end{proof}
        By Claim~\ref{c1}, we can find $N$ points $Q_1, \cdots, Q_N$ in the image of $\tau$ such that $Q_1, \cdots, Q_N$ are linearly independent. Take the set $S_d : = \{P_1 = \tau^{-1}(Q_1), \cdots, P_N = \tau^{-1}(Q_N)\}$. Let $\tau(S_d)$ is a matrix whose row vectors are $Q_1, \cdots, Q_N$. Then, it is invertible since $Q_1, \cdots, Q_N$ forma a basis:
        \[
        \tau(S_d) = \left( \begin{array}{c} \tau(P_1) \\ \vdots \\ \tau(P_N) \end{array} \right) = \left( \begin{array}{c} Q_1 \\ \vdots \\ Q_N \end{array} \right) \in \operatorname{GL}_N.
        \]

        Suppose $g := (g_{1}, \cdots, g_{n})$ where $g_i(X) = \sum a_{iJ} X^J$, $J=\{j_1, \cdots, j_n\}$ and $X^J = X_1^{j_1}\cdots, X_n^{j_n}$. Then, we have an equality between $n \times n $ matrices.
                \[
                {\tau}(X) \cdot \left(a_{iJ}\right)^t
                = \sum a_{iJ} X^J = g_i(X).
                \]
                Therefore,
                \[
                \tau(S_d) \cdot \left[\left(a_{1J}\right)^t, \cdots,  \left(a_{nJ}\right)^t \right]
                = \Bigl(
                    \sum a_{iJ}P_i^J
                \Bigr)_{i,j}
                = \Bigl(
                g_i(P_j)
                \Bigr)_{i,j}
                \]
              Since $\tau(S_d) $ is nonsingular, we can determine $\left[\left(a_{1J}\right)^t, \cdots,  \left(a_{nJ}\right)^t \right]$ uniquely.
    \end{proof}

    \begin{thm}\label{det2}
        Let $S$ be a Zariski dense subset of $\pp^n$. Then, for any $d\geq 1$, we can find a set of points in general position of degree $d$ in $S$.
    \end{thm}
    \begin{proof}
        Suppose not; then there is a hyperplane $H$ on $\pp^N$ such that $\tau(S) \subset H$. However, $\tau^{-1}(H)$ is a zero set of a linear combination of monomials of degree $d$, which is a proper algebraic subset of $\pp^n$. It contradicts to $S$ is Zariski dense.
    \end{proof}

    \begin{thm}\label{thmB}
    Let $f:\cc^n \rightarrow \cc^n$ be a polynomial map such that $\Pre(f)$ is of bounded height and Zariski dense. Then,
    \[
    \Com(f,d): = \{ g :\af^n \rightarrow \af^n ~|~ f \circ g =  g \circ f, ~\deg g = d\}
    \]
    is a finite set.
    \end{thm}
    \begin{proof}
        Let $S = \Pre(f)$. Then, by assumption, $S$ is Zariski dense so that we can find $S_d \subset S$. Moreover, $S_d$ is a finite set and $S = \bigcup_m \Pre_{m,l}(f)$ so that we can find $m_d$ such that $S_d \subset V_d : = \bigcup_{m=0}^{m_d} \bigcup_{l=0}^m \Pre_{m,l}(f)$. Proposition~\ref{cm} says $g(V_d) \subset V_d$ so that possible images of $S_d$ by $g\in \Com(f,d)$ are subsets of $V_d$. Therefore,
        \[
        \left|\Com (f,d) \right|  \leq \bigcup_{m=1}^{m_d} \bigcup_{l=0}^{m-1}  \bigl|\{ g : \Pre_{m,l}(f) \rightarrow \Pre_{m,l}(f) \}   \bigr| \leq \prod_{m=1}^{m_d} \prod_{l=0}^{m-1}  M_{m,l}^{M_{m,l}} < \infty
        \]
        where $M_{m,l}  = |\Pre_{m,l} (f)| $.
    \end{proof}

\section{Preperiodic points of polynomial maps of small $D$-ratio}

    Recall that an endomorphism of degree $d>1$ on $\pp^n$ satisfies that $\Pre(f)$ is Zariski dense and of bounded height. However, it is not true for polynomial maps in general. In such case, $\Com(f,d)$ can be an infinite set.

    \begin{ex}
        Let $f(x,y,z) = (y,x,z^2)$ and let $P(x,y)$ be a polynomial of degree $d$. Define
        \[
        g_P(x,y,z) := (P(x,y), P(y,x),z^d).
        \]
        Then,
        \[
        f \circ g_P = (P(y,x), P(x,y), z^{2d}) = g_P \circ f
        \]
        and hence $\Com(f,d)$ has uncountably many polynomial maps.
    \end{ex}

    \begin{prop}
    The following polynomial maps $f:\af^n \rightarrow \af^n$ have Zariski dense preperiodic points:
        \begin{enumerate}
            \item There is $m\geq 0$ such that $f^m$ is extended to an endomorphism $\widetilde{f^m} : \pp^n \rightarrow \pp^n$. \cite{Le2}.
            \item $f=(f_1, \cdots, f_n), f_i \in \cc[x_1, \cdots, x_i]$ for all $i= 1, \cdots, n$. \cite{MS}.
            \item $f$ is a regular polynomial automorphism. \cite{K2, Le}.
        \end{enumerate}
    \end{prop}

    Here's another example, if $f$ is  a polynomial map whose
    $D$-ratio is bounded by $\deg f$, then $\Pre(f)$ is Zariski
    dense.

\begin{thm}\label{Zdense}
    Let $K$ be a finitely generated field over $\mathbb{Q}$ and let $f: \af^n \rightarrow
\af^n$ be a polynomial map, defined over $K$, such that $r(f) <\deg f$. Then, the
set of preperiodic point is Zariski dense.
\end{thm}
\begin{proof}
    Note that we will prove only for number field case since we can prove the general case by induction on the transcendental degree: if we consider an integral model of $\pp^n$ over $\operatorname{Spec}\ox_K$, then every special fiber is defined over a field $K'$ such that $\operatorname{tr.deg}K' = \operatorname{tr.deg}K -1$.

    Suppose that $K$ is a number field and consider an integral model $\mathcal{A}$ of $\af^n$ over $\operatorname{Spec} \ox_K$.

    Let $\mathcal{U}\subset M_K$ be the set of prime places at which $f$ has good reduction and let $\mathcal{P}_{m,l}$ be the closure of $\Pre_{m,l}(f)$ in $\mathcal{A}$. Then, for all $v\in \mathcal{U}$, $\mathcal{A}_v \cap \mathcal{P}_{m,l}$ is not empty and hence $\mathcal{P}_{m,l}$ is a nontrivial subvarieties of of dimension at least 1. But, $\mathcal{A}_v \cap \mathcal{P}_{m,l}$ is a fintie set of points so that $\mathcal{P}_{m,l}$ consists of finitely many vertical fibers and finitely many horizontal sections. Therefore, $\bigcup \mathcal{P}_{m,l}$ contains infinitely many horizontal sections whose intersection with all but finitely many special fibers are whole fiber. Hence, $\bigcup \mathcal{P}_{m,l}$ consists of infinitely many horizontal sections whose intersection with all fiber, especially with the generic
fiber, is Zariski dense.
\end{proof}

    \begin{cor}\label{K}
        Let $f:\cc^n \rightarrow \cc^n$ be a polynomial map such that $r(f) < \deg f$. Then Theorem~\ref{thmB} holds for~$f$.
    \end{cor}

    \begin{ex}
    Let
    \[
    f(x,y) = (x^3+y, x+y^2).
    \]
    Then, $r(f) = \dfrac{3}{2} < 3 = \deg f$ so that $\Pre(f)$ is
    Zariski dense. Therefore, $\Com(f,d)$ is a finite set of all $d\geq 1$.
    \end{ex}

\end{document}